\documentclass[11pt]{article}
\usepackage[affil-it]{authblk}

\usepackage{amsmath,amssymb,amsthm,enumerate}
\usepackage{natbib}

\newtheorem{lemma}{Lemma}
\newtheorem{theorem}{Theorem}

\newcommand{\prob}{{\mathbb P}}

\begin{document}

\title{Maximal Steiner Trees in the Stochastic Mean-Field Model of Distance}
\author{A. Davidson%
\thanks{email: \texttt{angus.davidson@bristol.ac.uk}} and A. Ganesh%
\thanks{email: \texttt{a.ganesh@bristol.ac.uk}}}
\affil{School of Mathematics, University of
Bristol, University Walk, Bristol BS8 1TW}

\setcitestyle{numbers}

\maketitle

\begin{abstract}
Consider the complete graph on $n$ vertices, with edge weights drawn independently from the exponential distribution with unit mean. Janson showed that the typical distance between two vertices scales as $\log{n}/n$, whereas the diameter (maximum distance between any two vertices) scales as $3\log{n}/n$. Bollob\'{a}s et al. showed that, for any fixed k, the weight of the Steiner tree connecting $k$ typical vertices scales as $(k-1)\log{n}/n$, which recovers Janson's result for $k=2$. We extend this result to show that the worst case $k$-Steiner tree, over all choices of $k$ vertices, has weight scaling as $(2k-1)\log{n}/n$ and finally, we generalise this result to Steiner trees with a mixture of typical and worst case vertices.
\end{abstract}

\section{Introduction} \label{sec:intro}

Consider the complete graph on $n$ nodes, and assign weights to the edges, drawn independently from 
an exponential distribution with unit mean. This model (or equivalents in which the edge weights are i.i.d. 
with a distribution that has non-vanishing derivative at the origin), is known as the stochastic mean field 
model of distance. A number of combinatorial optimisation problems have been studied on this model. 
For example, Frieze~\citep{frieze1985value} showed that the minimal spanning tree has weight converging 
to $\zeta(3) = \sum_{n=1}^{\infty} n^{-3}$ in probability as $n$ tends to infinity, while van der Hofstad, 
Hooghiemstra and van Mieghem~\citep{van2006size} showed that the broadcast tree (tree of shortest paths 
to all nodes from a specified root) has weight converging to $\zeta(2)$ in probability. For the travelling 
salesman problem, it was shown by Frieze~\citep{frieze2004random} that the length of the optimal tour lies 
between $\zeta(3)$ and $6$ with high probability, and by W\"astlund~\citep{wastlund2010mean} that it 
converges in probability to a constant, for which an explicit expression is given in terms of an integral.
Bhamidi, van der Hofstad and Hooghiemstra \citep{bhamidi2011first} studied first passage percolation 
on the mean-field model, as well as on the Erd{\H{o}}s--R{\'e}nyi random graph with random edge weights. 

We now introduce the notation we will use in the rest of the paper, and state our problem more formally.
Let $G = (V,E)$ be the complete undirected graph on $n$ vertices, labelled $v_1, v_2,.. v_n$, with i.i.d. 
edge weights having the $Exp(1)$ distribution. We use $T_{ij}$ to denote the weight of the edge between 
nodes $v_i$ and $v_j$. The ``Steiner Tree Problem'' for a set of vertices $S \subset V$ concerns finding the minimum weight connected subgraph of $G$ containing $S$. This will naturally be a tree, and will be almost surely unique. We will use $w(S)$ to denote the weight of this tree, which is a random variable. Of course, we get very different distributions for $w(S)$ when $S$ is a ``typical'' set of vertices, (which might as well be fixed at the outset as $\{v_1,v_2,..v_k\}$) compared to $S$ being a ``worst case'' set of vertices, (where we maximise $w(S)$ over all sets $S$ of a certain size). Let us define
$$
W_{k,l} := \max_{S \subset V}(w(S): v_1,v_2,..v_k\ \in S, |S| = k+l).
$$
So $W_{k,l}$ is the minimum weight of a tree connecting $k+l$ vertices, $k$ of which are ``typical'', maximised over the choice of the $l$ remaining vertices. 

Janson proved the following results \citep{janson1999one} for geodesics between pairs of points on $G$, (i.e. Steiner trees connecting $2$ points):
\begin{theorem}[Janson '99]\label{thm:Jan} 
As the number of vertices, $n$, tends to infinity, the following hold:
\begin{enumerate}[(i)]
\item
$$\frac{W_{2,0}}{\log{n}/n} \stackrel{p}{\to} 1$$
\item
$$\frac{W_{1,1}}{\log{n}/n} \stackrel{p}{\to} 2$$
\item 
$$\frac{W_{0,2}}{\log{n}/n} \stackrel{p}{\to} 3$$
\end{enumerate}
\end{theorem}
The first of these results shows the typical weight of a Steiner tree connecting $2$ vertices is $\log{n}/n$ and the last shows that the diameter of G is with high probability $3\log{n}/n$.

Bollob\'{a}s et. al. \citep{bollobas2004value} proved a generalisation of the first of these results for Steiner trees on 2 or more vertices. We state below a less general version of the result they proved which fits more neatly into our context:
\begin{theorem}[Bollob\'{a}s et. al. '04] \label{thm:Bollobas}
For any positive integer $k$,
$$\frac{W_{k,0}}{\log{n}/n}\stackrel{p}{\to} (k-1) \text{ as } n \to \infty$$
\end{theorem}

In other words, the typical weight of a Steiner tree connecting $k$ vertices is $(k-1)\log{n}/n$. 
In this paper we show the following generalisation of Janson's 3rd result concerning the diameter of $G$ (Theorem \ref{thm:Jan} (iii)):
\begin{theorem} \label{thm:main}
For any positive integer $k$,
$$\frac{W_{0,k}}{\log{n}/n} \stackrel{p}{\to} (2k-1) \text{ as } n \to \infty$$
\end{theorem}
We further generalise it to establish the result below, which follows from Theorem \ref{thm:Jan}, 
Theorem \ref{thm:Bollobas} and an adapted version of the lower bound proof of Theorem \ref{thm:main}:
\begin{theorem} \label{thm:general}
For any positive integers $k$ and $l$,
$$\frac{W_{k,l}}{\log{n}/n}\stackrel{p}{\to} (k+2l-1) \text{ as } n \to \infty$$
\end{theorem}

\section{Proofs} \label{sec:proofs}

We split the proof of Theorem 3 into two parts, considering the upper and lower bounds separately. 

\subsection{Upper Bound}

We prove the upper bound by creating a random variable which dominates $W_{k,0}$ in Theorem 2, and applying a Chernoff bound to this dominating random variable. The upper bound is provided by an explicit algorithm to construct a spanning tree for $k$ specified vertices; by definition, the weight of such a spanning tree is an upper bound on the weight of the Steiner tree connecting these vertices.

For convenience we define $c_{k,n} = n^{\frac{k-1}{k}}\sqrt[k]{(k+1)\log{n}}$.

Heuristically we can think of this half of the proof as follows: With high probability, for \textbf{any} $k$ given starting vertices there is a vertex in $G$ within distance $(2-\frac{1}{k})\log{n}/n$ of all of them. Put another way, given any $k$ vertices in $G$, we can ``grow (metric) balls'' around the vertices until they each encompass $c_{k,n}$ vertices, at which point these balls will have a non-trivial intersecting set. The radius of any such ball in $G$ (around any vertex in $V$) is bounded above by $(2-\frac{1}{k})\log{n}/n$ with high probability (w.h.p), i.e., with probability tending to $1$ as $n$ tends to infinity. Hence, we can find a tree connecting a chosen vertex in the intersection of the balls and the $k$ starting vertices of weight at most $(2k-1)\log{n}/n$ by taking the union of the paths from the starting vertices to the chosen vertex. We can think of $c_{k,n}$ as $n^{\frac{k-1}{k}}$ multiplied by the right factor to ensure our balls are big enough to have a non-trivial intersection with high probability, but not big enough to increase their radius significantly.

If we simply grow these balls about the $k$ points in the standard way however, the calculations are complicated immensely by the dependencies on balls with overlapping vertices, so we grow the balls in a restricted fashion to avoid this complication.

We now describe the random variable which we will use to dominate $W_{k,0}$ in Theorem 2. We fix at outset $S := \{v_1,v_2,..v_k\}$ to be our $k$  starting vertices. The ``ball of radius $t$ about vertex $v_i$'' can be described via an infection model as vertices ``infected'' within time $t$ starting from a single infected vertex, $v_i$, and where the weight of an edge denotes the time it takes for the infection to pass from the vertex at one end to that at the other. In other words, the edge length $T_{ij}$ is the time taken for $v_i$ to directly infect $v_j$ once $v_i$ is infected, (and infecting an already infected vertex has no effect). Hence, vertex $v_k$ is in the infected set at time $t$ or equivalently ball of radius $t$ about $v_i$ if and only if there is a path of length no more than $t$ connecting $v_i$ and $v_k$.  

{\bf Algorithm:} The algorithm takes as input a set of specified vertices $S := \{v_1,v_2,..v_k\}$, and the edge weights of the complete graph. It outputs a spanning tree of the vertices in $S$ w.h.p.; with the residual probability, the algorithm fails and terminates with empty output. The algorithm proceeds in stages as described below.

\textbf{Stage 1.1}
Define $\widetilde{V}_1:=V\setminus \{v_2,..v_k\}$ and let $G|_{\widetilde{V}_{1}}$ denote
the subgraph of G induced by the vertex set $\widetilde{V}_{1}$. Consider an infection starting 
from the single ``infected'' vertex $v_1$, and spreading along the edges in $G|_{\widetilde{V}_{1}}$ 
with independent Exp(1) waiting times. We let infection spread until $c_{k,n}$ vertices are infected. 
Call this set of infected vertices $V_{1}^{1}$. 

Note that none of the vertices $\{v_2, \ldots, v_{k}\}$ are contained within the ball of infection grown 
around vertex $v_1$. Next, we will grow similar balls of infection around the remaining vertices 
$v_2,\ldots,v_k$, excluding all previously grown balls.


\textbf{Stage 1.i} For $2\le i\le n$, define $\widetilde{V}_{i} :=\{ \widetilde{V}_{i-1} \cup v_i \} 
\setminus V_{i-1}^{1}$, and let $G|_{\widetilde{V}_{i}}$ denote the subgraph of $G$ induced by the 
vertex set $\widetilde{V}_{i}$. Starting from a single ``infected'' vertex $v_i$, infection spreads 
independently along each edge in $G|_{\widetilde{V}_{i}}$ with independent Exp(1) waiting times. 
Exactly as for stage 1.1, the infection spreads until $c_{k,n}$ vertices are infected. Call the 
infected set $V_{i}^{1}$.

The crucial point to note is that all vertices seen in previous stages have been removed, and therefore so have all edges incident to them. Consequently, we have no information from prior stages about any of the edges being used in stage 1.i. Hence, the weights on these edges are indeed i.i.d. Exp(1) random variables, as claimed.

In the next step, we blow up the balls $\{ V_{i}^{1} \}_{i=1}^{n}$ of size $c_{k,n}$ grown around each vertex to balls of size $2c_{k,n}$, but using only single hops, i.e., edges rather than paths. In other words, we add an annulus of size $c_{k,n}$ to each ball. These annuli are all grown inside the same set $V \setminus \bigcup_{i=1}^{n} V_{i}^{1}$; vertices which were uninfected after the stage 1. As the balls $V_{i}^{1}$ grown at the previous stage were vertex disjoint, and this stage only uses single edges, the edges used to connect vertices (possibly the same vertex) to different balls will be distinct.

\textbf{Stage 2.i:} With $c_{k,n}$ vertices infected in stage i.1, ($1 \le i \le n$), we let the infection \textit{continue} to spread only along edges incident to $V_{i}^{1}$ to vertices in $V \setminus \bigcup_{i=1}^{n} V_{i}^{1}$. We note that the \textit{remaining} length of each such edge is Exp(1) by the memoryless property of Exponential distribution. Stage 2.i is completed once $c_{k,n}$ further vertices are infected. Call the set of vertices infected during this stage $V_{i}^{2}$.


\textbf{Construction of a spanning tree:} If the intersection of the sets $V_{i}^{2}$ constructed in Stage i.2 is 
non-empty, then pick an arbitrary vertex $w \in \bigcap_{i=1}^{k} V_{i}^{2}$ and, for each $j$ between $1$ and $k$, define $P_j$ as the (a.s. unique) minimum length path between $w$ and $v_j$. Define the length of time to execute stage $i.j$ as $Z_{i}^{j}$, and $Z_{i} := Z_{i}^{1}+Z_{i}^{2}$. The length of the path $P_j$ will necessarily be less than $Z_{j}$ (because $w \in V_{j}^{2}$). Define $\mathcal{T} = \bigcup_{i=1}^k P_i$. Then, $\mathcal{T}$ is a graph connecting $\{v_1,v_2,..v_k\}$. If $\bigcap_{i=1}^{k} V_{i}^{2} =\emptyset $, then declare the algorithm to have failed. We will provide an upper bound on the failure probability. Clearly, the total weight of the edges in $\mathcal{T}$ dominates $W_{k,0}$, and hence so does $\sum_{i = 1}^{n} {Z_i}$.

A similar construction appears in a paper of Bhamidi and van der Hofstad~\citep{bhamidi2013diameter} where they use it to obtain the joint distribution of pairwise distances between $k$ typical vertices. 

The following lemma will be needed. We adapt the proof of the upper bound of Theorem 1 which appears in Janson's paper \citep{janson1999one}.

\begin{lemma} For fixed $0 < \varepsilon < 1$, and $t := (1 - \frac{1}{\log{n}})(1-\varepsilon )$;
$$\mathbb{E}(e^{ntZ_{1}}) = O(c_{k,n})$$
\end{lemma}

\begin{proof}
As in the infection spreading model, we can think of $Z_{1}^{1}$ as the first time an infection spreads to a set 
of vertices of size $c_{k,n}$ starting from an initial starting vertex, ($v_1$), in $G|_{\widetilde{V}_{1}}$, where 
we define $n' := |\widetilde{V}_{1}| = n - k + 1$. If vertex $v_i$ is infected before vertex $v_j$, we think of the 
edge length $T_{ij}$ as the time taken for $v_i$ to directly directly infect $v_j$. We define $X_i, 1 \leq i \leq c_{k,n}-1$, 
as the time between $i$ vertices being infected and $i+1$ vertices being infected. Then, $X_1$ is the minimum 
of $(n - k) = (n'-1)$ exponential mean $1$ random variables corresponding to the edges between $v_1$ and the 
nodes in $\tilde V_{1}$. Hence, $X_1$ is distributed as Exp($n'-1$). 

Denote the vertex infected at time $X_1$ by $u_1$. At this time, there are two infected nodes and $n'-2$ uninfected 
nodes in $\tilde V_1$. The $n'-2$ edges between $u_1$ and the uninfected nodes have i.i.d. Exp(1) lengths. The 
lengths of all edges between $v_1$ and the uninfected nodes are necessarily bigger than $X_1$; moreover, by the 
memoryless property of the exponential distribution, they exceed $X_1$ by random amounts which are also i.i.d. Exp(1) random variables. Hence $X_2$, the additional time then taken for the 3rd vertex to be infected, will be the minimum of the $2(n'-2)$ independently distributed Exp(1) random variables. Hence $X_2 \sim$ Exp($2(n'-2)$). Likewise, $X_i\sim$ Exp($i(n'-i)$). Moreover, the random variables $X_1,X_2,\ldots$ are mutually independent, and
$$
Z^1_1= X_1+X_2+\ldots+X_{c_{k,n}-1}.
$$

Next, $Z_{1}^{2}$ is independent of $Z_{1}^{1}$, again by the memoryless property of the exponential 
distribution. Indeed, if node $u_i$ was infected at time $t_{u_i} < Z_{1}^{1}$ and node $u_j$ was not 
infected before time $Z_{1}^{1}$, then the residual edge length $T_{ij} -(Z_{1}^{1} - t_{u_i})$ has an 
Exp(1) distribution, and all the residual edge lengths are mutually independent. If we define $X'_i$, for 
$1 \leq i \leq c_{k,n}$, as the time between the $(i-1)$th vertex being infected and the $i$th vertex being 
infected \textbf{after} time $Z_{1}^{2}$, then $X'_i$ is the minimum of $(n'-kc_{k,n}-(i-1)) c_{k,n}$ independent 
Exp(1) random variables, and has an Exp($(n'-kc_{k,n}-i+1)c_{k,n}$) distribution. Note that we subtract $kc_{k,n}$ from $n'$ above since in stage 2.1 we remove $\bigcup_{i=1}^{n} V_{i}^{1}$ from $V$.

Now, for $-\infty \leq t < 1-1/n'$, we have
$$
\mathbb{E}e^{ntZ_{1}^{1}} = \prod_{i=1}^{(c_{k,n}-1)}\mathbb{E}e^{ntX_i} = 
\prod_{i=1}^{(c_{k,n}-1)}\left(1-\frac{nt}{i(n'-i)}\right)^{-1}.
$$
Fix $0 < \varepsilon < 1$ and set $t = (1-\frac{1}{\log{n}})(1-\varepsilon)$. Then, using the inequality 
$-\log{(1-x)} \leq x+x^2$, which holds for all $x \in \lbrack 0,3/5\rbrack $, we obtain that
$$
\begin{aligned}
\mathbb{E}e^{ntZ_{1}^{1}} &= \Bigl( 1-\frac{nt}{n'-1} \Bigr)^{- 1} 
\exp{ \Bigl( \sum_{i=2}^{(c_{k,n}-1)} -\log{  \Bigl( 1-\frac{nt}{i(n'-i)} \Bigr) } \Bigr) } \\
&\leq \Bigl( 1-\frac{nt}{n'-1} \Bigr)^{-1} 
\exp{ \Bigl( \sum_{i=2}^{c_{k,n}} \Bigl[ \frac{nt}{i(n'-i)} + \Bigl( \frac{nt}{i(n'-i)} \Bigr)^{2} \Bigr] \Bigr) } \\
&= \Bigl( 1-\frac{nt}{n'-1} \Bigr)^{-1}
\exp{ \Bigl( \sum_{i=2}^{c_{k,n}} \Bigl[ \frac{n}{n'} \Bigl( \frac{t}{i}+\frac{t}{n'-i} \Bigr) +
\Bigl( \frac{nt}{i(n'-i)} \Bigr)^{2} \Bigr] \Bigr) } \\
&= \left(1-t+O(n^{-1})\right)^{-1}\exp{\left(\frac{n}{n'}t\log{(c_{k,n})}+O(1)\right)}\\
&= O(c_{k,n}/\varepsilon) = O(c_{k,n}).
\end{aligned}
$$
Similarly, and still with $t = (1-1/\log{n})(1-\varepsilon )$, we obtain for large enough $n$ that,
$$
\begin{aligned}
&\mathbb{E}e^{ntZ_{1}^{2}} \; = \; \exp{ \biggl( \sum_{i=1}^{c_{k,n}} 
-\log{ \Bigl(( 1-\frac{nt}{c_{k,n}(n-kc_{k,n}-i+1)} \Bigr) } \biggr) } \\
&\leq \exp{ \biggl( \sum_{i=1}^{c_{k,n}} \Bigl[ \frac{nt}{c_{k,n}(n-kc_{k,n}-i+1)}+
\Bigl( \frac{nt}{c_{k,n}(n-kc_{k,n}-i+1)} \Bigr)^{2} \Bigr] \biggr) } \\
&\leq \exp{\Bigl( \frac{c_{k,n}nt}{c_{k,n}(n-(k+1)c_{k,n})}+O(1) \Bigr) } \\
&\leq \exp{ \bigl( 2+O(1) \bigr) }. \\
\end{aligned}
$$
Combining the above estimates, we obtain
$$\mathbb{E}e^{ntZ_{1}} = \mathbb{E}e^{ntZ_{1}^{1}} \cdot \mathbb{E}e^{ntZ_{1}^{2}} = O(c_{k,n})$$
This concludes the proof of the lemma.
\end{proof}

More generally, for all $i \in \{ 1,\ldots,k \}$, for fixed $0 < \varepsilon < 1$, and 
$t = (1 - \frac{1}{\log{n}})(1-\varepsilon )$, one can show that
$$\mathbb{E}(e^{ntZ_{i}}) = O(c_{k,n}),$$ 
using the same proof as above but with $n' := |\widetilde{V}_{i}|$.

The following lemma will be used to show that the random sets $V_{i}^{2}$ have non-empty intersection w.h.p. 

\begin{lemma}
Let $A_i, 1 \leq i \leq k$ be independent uniformly chosen subsets of $[n]$ of size $m \leq n$. Define $B = \bigcap^{k}_{i=1} A_i$. Then 

$$\prob(B = \emptyset) \leq e^{-m^{k}n^{1-k}}$$
\end{lemma}

\begin{proof}
The main step in the proof involves a coupling, so we first define the related random objects we will couple with:

Let ${A'}_i, 1 \leq i \leq k$ be independent uniformly chosen subsets of $[n-1]$ of size $m-1$, (so here our set is 1 element smaller, and all subsets are also 1 element smaller). Define ${B'} = \bigcap^{k}_{i=1} {A'}_i$.

We note that for all $2 \leq j \leq n$
$$
\begin{aligned}
\prob\left(\{j \} \in B\Bigm|[j-1] \subset B^C\right) &= \frac{\prob\left([j-1] \subset B^C\Bigm|\{j \} \in B\right)}{\prob\left([j-1] \subset B^C\right)}\prob\left(\{j \} \in B\right)\\
&= \frac{\prob\left([j-1] \subset {B'}^C\right)}{\prob\left([j-1] \subset B^C\right)}\prob\left(\{j \} \in B\right)
\end{aligned}
$$
The last line follows since $\{j \} \in B$ means $\{j \} \in A_i$ $\forall$ $1 \leq i \leq k$, and so the remaining $m-1$ elements of $A_i$ will be a uniformly chosen subset of $[n] \backslash \{j \}$, independent for each $i$.

Now we couple $B$ and ${B'}$ as follows. For each $1 \leq i \leq k$ we first realise $A_i$ (uniformly in $[n]$), and then derive ${A'}_i$ by eliminating ${n}$ from $A_i$ if ${n} \in A_i$ or a uniformly random element of $A_i$ if not. Clearly ${A'}_i$ are i.i.d. and have the required distribution. Further, for this construction ${B'} \subseteq B$. Hence
$$
\prob\left([j-1] \subset {B'}^C\right) \geq \prob\left([j-1] \in B^C\right)
$$
and so 
$$
\prob\left(\{j \} \in B\Bigm|[j-1] \subset B^C\right) \geq \prob\left(\{j \} \in B\right)
$$

Finally,
$$
\begin{aligned}
\prob(B = \emptyset) &= \prob(\{1 \} \in B^C)\prod_{j=2}^{n}{\prob\left(\{j \} \in B^C \Bigm|[j-1] \subset B^C\right)}\\
&\leq \prod_{j=1}^{n}{\prob(\{j \} \in B^C)}\\
&= (1-\prob(\{j \} \in B))^n = \left(1-\left(\frac{m}{n}\right)^k\right)^n \leq e^{(-m^{k}n^{1-k})}
\end{aligned}
$$
\end{proof}

The proof of the upper bound is completed using the following lemma. 

\begin{lemma}
$\forall \, \varepsilon > 0$, as $n \to \infty$:
$$
\prob(W_{0,k} > (2k-1+(2k+1)\varepsilon)\log{n}/n) \to 0.
$$
\end{lemma}

\begin{proof}
We first use Lemma 2 to show that the algorithm described earlier succeeds w.h.p.. Next, conditional on the event that it succeeds, we use Lemma 1 to obtain a Chernoff bound on the weight of the tree ${\cal T}$ that it constructs, (which clearly dominates $W_{k,0}$, as $W_{k,0}$ has minimum weight among all graphs connecting $\{ v_1,v_2,\ldots, v_k \}$). We finish the argument using a union bound.

Note that by construction:

$$|V_{i}^{2}| = c_{k,n}, 1 \leq i \leq k$$
where $V_{i}^{2}$ are all uniformly chosen subsets of $V(G)\setminus \bigcup_{j=1}^{k} V_{j}^{1}$, and ($\forall \varepsilon \geq 0$) 

$$|V(G)\setminus \bigcup_{j=1}^{k} V_{j}^{1}| = n - kc_{k,n} > n(1-\varepsilon)$$
for big enough $n$. Clearly in the statement of Lemma 2 the given probability is monotone decreasing in the size of the uniformly chosen subsets, (so gives the upper bound we require in this setting). Applying this lemma: 

$$
\begin{aligned}
\prob(\bigcap_{i = 1}^{k} V_{i}^{2} = \emptyset) &\leq e^{-c_{k,n}^k(n(1-\varepsilon))^{1-k}} \\
&= e^{-(k+1)(\log{n}) n^{k-1}n^{1-k}(1-\varepsilon)^{1-k}} \\
&\leq e^{-(k+1)\log{n}} \\
&= n^{-(k+1)}
\end{aligned}
$$

Now we apply a Chernoff bound using Lemma 1. 
Fix $t = (1-1/\log{n})(1-\varepsilon )$:
$$
\begin{aligned}
&\prob(\{ \text{algorithm fails} \} \cup w(\mathcal{T}) \geq (2k-1+(2k+1)\varepsilon)\log{n}/n) \\
&\leq \prob(\text{algorithm fails}) + \prob(w(\mathcal{T}) \geq (2k-1+(2k+1)\varepsilon)\log{n}/n) \\
&\leq \prob(\bigcap_{i = 1}^{k} V_{i}^{2} = \emptyset) + \mathbb{E}(e^{ntw(\mathcal{T}) - t(2k-1+(2k+1)\varepsilon)\log{n}})\\
&\leq n^{-(k+1)} + e^{-t(2k-1+(2k+1)\varepsilon)\log{n}}\mathbb{E}(e^{ntw(\mathcal{T})})\\
&\leq n^{-(k+1)} + e^{-t(2k-1+(2k+1)\varepsilon)\log{n}}\prod_{i=1}^{k}{\mathbb{E}(e^{ntZ_{i}})}\\
&= O(n^{-(k+1)} + n^{-t(2k-1+(2k+1)\varepsilon)}(c_{k,n})^k)\\
&= O(n^{-(k+1)} + n^{-(2k-1+\varepsilon)}(k+1)(\log{n})n^{k-1})\\
&= O(n^{-(k+\varepsilon)}\log{n})
\end{aligned}
$$
In the second last line we have used that $\varepsilon < \frac{1}{2k+1}$. We now have
$$
\prob(W_{k,0} \geq (2k-1+(2k+1)\varepsilon)\log{n}/n) = O(n^{-(k+\varepsilon)}\log{n})
$$

We take the union bound of the above probability (relating to k typical vertices) over all subsets of G of size $k$, and so $\forall \, 0 < \varepsilon < \frac{1}{2k+1}$;
$$
\begin{aligned}
\prob(W_{0,k} \geq (2k-1+(2k+1)\varepsilon)\log{n} /n) &= O\left( \binom{n}{k} n^{-(k+\varepsilon)}\log{n}\right) \\
&= O(n^{-\varepsilon}\log{n})
\end{aligned}
$$

We have proved the lemma for $\varepsilon < \frac{1}{2k+1}$, but it is clear that the event in the lemma statement is decreasing in $\varepsilon$, and so the result holds $\forall \, \varepsilon > 0$.

\end{proof}
This concludes the proof of the upper bound.

\subsection{Lower Bound}

We now turn our attention to the lower bound. This entire subsection is an extension of the proof offered in Janson's paper \citep{janson1999one} for the equivalent lower bound result in Theorem 1 (iii), (i.e. the steiner tree on 2 vertices). We follow his proof very closely, but require some slightly uglier notation for the generalisation.

Fix $\varepsilon > 0$. Partition the vertex set $\{v_{1}, ..., v_{n}\}$ of our graph $G$ into the sets $A_1 = \{v_{1}, ..., v_{n_A}\}$, $A_2 = \{v_{n_A +1}, ..., v_{2n_A}\}$, ..., $A_k = \{v_{(k-1)n_A +1}, ..., v_{kn_A}\}$ and $B = \{v_{kn_A +1}, ..., v_{n}\}$, where $n_A = \lceil n^{1-\varepsilon}\rceil$. Let $A = \bigcup_{i \leq k} A_i$ and $n_B = n-kn_A$.

For $v_i \in V, i\leq kn_A$, let $U_i = \text{min}_{v_j \in B} T_{ij}$. Then ($U_i$, $i\leq kn_A$) are independent with $U_i \sim Exp(n_B)$. In particular,

\begin{align*}
\mathbb{P}(U_i > (1-2\varepsilon)\log{n}/n) 
&= \exp(-(1-2\varepsilon)\frac{n_B}{n}\log{n}) \\
&\geq \exp(-(1-2\varepsilon)\log{n}) = n^{2\varepsilon-1}
\end{align*}

and thus
$$\mathbb{P}(U_i \leq (1-2\varepsilon)\log{n}/n \text{ for every } v_{i}\in A_1) \leq (1-n^{2\varepsilon-1})^{n^{1-\varepsilon}} < e^{-n^{\varepsilon}}.$$

Let, for $v_{l_{j}} \in A_j$, $\mathcal{E}^{(j)}_{l_{j}}$ be the event that $U_{l_{j}} > (1-2\varepsilon)\log{n}/n$ but $U_i \leq (1-2\varepsilon)\log{n}/n$ for $i < {l_{j}}, v_{i} \in A_j$. Then for a fixed $j$ the events $\mathcal{E}^{(j)}_{l_{j}}$ are disjoint and, by the above,

$$\sum_{v_{l_{j}} \in A_j}\mathbb{P}(\mathcal{E}^{(j)}_{l_{j}}) = \mathbb{P}\left(\bigcup_{v_{l_{j}} \in A_j} \mathcal{E}^{(j)}_{l_{j}}\right) > 1-e^{-n^{\varepsilon}}$$

which implies

\begin{equation*}
\mathbb{P}\left(\bigcap_{j \leq k}\bigcup_{v_{l_{j}} \in A_j} \mathcal{E}^{(j)}_{l_{j}}\right) > 1-ke^{-n^{\varepsilon}} \tag{1} 
\end{equation*}

The idea behind the proof is to show that conditioned on $\mathcal{E}^{(j)}_{l_{j}}$, the weight of a typical minimum weight steiner tree containing vertex $v_{l_{j}}$ is increased by $(1-2\varepsilon)\log{n}/n$, while conditioning on $U_i < (1-2\varepsilon)\log{n}/n \text{ for }i<l_{j}, v_{i}\in A_j$ hardly affects the result. If we have $k$ such events, (one for each set $A_j$), the minimum weight Steiner tree containing the corresponding $k$ vertices will typically increase in weight by $k(1-2\varepsilon)\log{n}/n$. It is from this and Theorem 2 that we deduce our desired lower bound.

We will use the following lemma.

\begin{lemma}

Suppose that $\mu, b > 0$ and $X \sim Exp(\mu)$, and define
$$f(x) = -\mu \log{(e^{-b/\mu}+(1-e^{-b/\mu})e^{-x/\mu})}.$$

\begin{enumerate}[(i)]
	\item The distribution of $f(X)$ equals the conditional distribution of $X$ given $X \leq b$.
	\item If further $0 \leq \alpha < 1$ and $b/\mu \geq \alpha(1-\log{\alpha})/(1-\alpha)$, then $f(x) \geq \alpha x$ when $0 \leq x \leq \alpha^{-1}b-\mu$. Consequently,
\end{enumerate}

$$\prob(f(X) \leq \alpha X) \leq \prob(X > \alpha^{-1}b - \mu) = e^{1-\alpha^{-1}b/\mu}$$

\end{lemma}
\begin{proof}
We may for simplicity, by homogeneity, assume that $\mu = 1$. Then $e^{-X}$ is uniformly distributed on $[0,1]$, and thus for $0 \leq t \leq b$,

\begin{align*}
\prob(f(X) \leq t) &= \prob(e^{-b}+(1-e^{-b})e^{-X} \geq e^{-t}) = \prob(e^{-X} \geq \frac{e^{-t}-e^{-b}}{1-e^{-b}})\\
&= \frac{1-e^{-t}}{1-e^{-b}} = \prob(X \leq t | X \leq b),
\end{align*}
which proves (i).

For (ii) we observe that (when $\mu = 1$) $f(x) \geq \alpha x$ if and only if
\begin{equation*}
e^{-b}+(1-e^{-b})e^{-x} \leq e^{-\alpha x}. \tag{2}
\end{equation*}
Letting $y = e^{-x}$, the left hand side of (2) is a linear function of y, while the right hand side $y^{\alpha}$ is concave; hence, in order to verify (2) for the interval $0 \leq x \leq \alpha^{-1}b - 1$, it suffices to verify it for the endpoints.

For $x=0$, equation (2) is a trivial identity, while for $x = \alpha^{-1}b-1$, it is
\begin{equation*}
e^{-b}+(1-e^{-b})e^{-\alpha^{-1}b+1} \leq e^{-b+\alpha}. \tag{3}
\end{equation*}
Now, by assumption, $\alpha^{-1}b = b+b(1-\alpha)\alpha^{-1} \geq b+1-\log{\alpha}$, and thus
$$e^{-b}+e^{-\alpha^{-1}b+1} \leq e^{-b}+e^{-b+\log{\alpha}} = (1+\alpha)e^{-b} \leq e^{\alpha}e^{-b};$$
this implies (3), which completes the proof of the lemma.
\end{proof}

Continuing with the proof of the lower bound of our theorem, let $v_{l_{j}} \in A_j$ be fixed, let $f$ be as in Lemma 3 with $\mu = 1/n_B$ and $b = (1-2\varepsilon)\log{n}/n$, and for $v_{i} \in A_j$ define
\begin{equation*}
U'_{i}= 
   \begin{cases}
   f(U_{i}), &i<l_{j},\\
   U_{i}+b, &i=l_{j},\\
   U_{i}, &i>l_{j}.
   \end{cases}
\end{equation*}
Then, by Lemma 3, for $i<l_{j}$ and the standard lack-of-memory property of exponential distributions for $i=l_{j}$, the distribution of $U'_{i}$ equals the conditional distribution of $U_{i}$ given $\mathcal{E}^{(j)}_{l_{j}}$ for every $v_{i} \in A_j$; moreover, by our independence assumptions, this extends to the joint distribution. Furthermore, by the same lack-of-memory property, the family of random variables $\{ T_{im}-U_{i} \}_{m \in B}$ is independent of $U_{i}$, for each $v_{i} \in A_j$ separately, and thus for all $v_{i} \in A_j$ jointly too; hence the joint distribution of $\{ T_{im}-U_{i} \}_{v_{i} \in A_j, v_{m} \in B}$ is not affected by conditioning on $\mathcal{E}^{(j)}_{l_{j}}$. It follows that if we define $T'_{im}$ for $v_{i},v_{m} \in V, i < m$ by
\begin{equation*}
T'_{im}= 
   \begin{cases}
   T_{im}-U_{i}+U'_{i}, &v_{i} \notin B \text{ and } v_{m} \in B,\\
   T_{im}, &\text{otherwise},
   \end{cases} \tag{4}
\end{equation*}
and let $T'_{mi} = T'_{im}$ for $m>i$, then the family $\{ T'_{im} \}$ has the same distribution as the conditional distribution of $\{ T_{im} \}$ given $\bigcap_{1 \leq j \leq k}\mathcal{E}^{(j)}_{l_{j}}$, ($i \leq kn_A$). Note in particular that for $1 \leq j \leq k$, $T'_{{l_{j}}m} = T_{{l_{j}}m}+b$ when $v_{m} \in B$.

Suppose that for $1 \leq j \leq k$, $\{ T_{im} \}$ are such that
\begin{align*}
U'_{i} &\geq (1 - 2\varepsilon)U_{i} \text{ for every } i \in A,\tag{5} \\
T_{i{l_{j}}} &\geq (2k-1)\frac{\log{n}}{n} \text{ for every } i \in A \tag{6}
\end{align*}
and
\begin{equation*}
w(S) > (k-1-\varepsilon)\log{n}/n \text{ where } S = \{v_{l_{j}} : 1 \leq j \leq k\} \tag{7}
\end{equation*}

We observe first that, by (4) and (5), then
\begin{equation*}
T'_{im} \geq (1 - 2\varepsilon)T_{im} \text{ for every } i \text{ and } j \neq i. \tag{8}
\end{equation*}

Now consider $w'(S)$ which we define as the minimum weight steiner tree connecting the set S defined by the edge weights $T'_{im}$. By (7), the minimum weight steiner tree connecting $S$ has weight $w(S) \geq (k-1-\varepsilon)\log{n}/n$. Consider such a tree, and the corresponding weight $w'(S)$. Either there is a leaf of the minimum weight tree sharing an edge with a vertex outside $B$, and then, by (4) and (6), $w'(S) \geq (2k - 1)\log{n}/n$ or all leaves share edges with vertices inside $B$ which, together with (8) yields

\begin{align*}
w'(S) &\geq kb + (1-2\varepsilon)w(S)\\
&\geq k(1-2\varepsilon)\frac{\log{n}}{n}+(1-2\varepsilon)(k-1-\varepsilon)\frac{\log{n}}{n}\\
&\geq (2k-1-(4k-1)\varepsilon)\frac{\log{n}}{n}.
\end{align*}

We have shown that if (5)-(7) hold, then $w'(S) \geq (2k-1-(4k-1)\varepsilon)\log{n}/n$. Consequently,
\begin{align*}
&\prob(w(S) \geq (2k-1-(4k-1)\varepsilon)\log{n}/n | \bigcap \mathcal{E}^{(j)}_l) \\
&= \prob(w'(S) \geq (2k-1-(4k-1)\varepsilon)\log{n}/n)\\
&\geq \prob(\text{(5)-(7) hold}).
\end{align*}

Let q denote the probability that (5)-(7) hold. We have so far kept $S = \{v_{l_{j} : 1 \leq j \leq k}\}$ fixed, but q is independent of $S$, and summing over the choices for S we obtain
\begin{align*}
&\text{   }\prob(\max_{\substack{S \subset V \\ |S| = k}} w(S) \geq (2k-1-(4k-1)\varepsilon)\log{n}/n)\\
&\geq \sum_{v_{l_{j}} \in A_j ; 1 \leq j \leq k} \prob(w(\{v_{l_{j} : 1 \leq j \leq k}\}) \geq (2k-1-(4k-1)\varepsilon)\log{n}/n | \bigcap \mathcal{E}^{(j)}_{l_j})\prob(\bigcap \mathcal{E}^{(j)}_{l_j})\\
&\geq q\sum_{v_{l_{j}} \in A_j ; 1 \leq j \leq k} \prob(\bigcap \mathcal{E}^{(j)}_{l_j}) = q\mathbb{P}(\bigcap_{j \leq k}\bigcup_{v_{l_{j}} \in A_j} \mathcal{E}^{(j)}_{l_{j}}).  \tag{9}
\end{align*}

Now, by Lemma 3(ii) with $\alpha = 1-2\varepsilon$, if n is large enough,
\begin{align*}
\prob(\text{(5) fails}) &\leq \sum_{i \in A} \prob(U'_{i} < (1-2\varepsilon)U_{i}) \leq kn_{A}e^{1-n_{B}\log{n}/n}\\
&= O(kn^{1-\varepsilon}n^{-1}) = o(1).
\end{align*}
Similarly,
$$\prob(\text{(6) fails}) \leq k\sum_{i \in A} \prob(T_{il_{1}} < (2k-1)\frac{\log{n}}{n}) \leq k(kn_{A}(2k-1)\frac{\log{n}}{n}) = o(1)$$
while $\prob(\text{(7) fails}) = o(1)$ by Theorem 2.

Consequently, $q = 1-o(1)$, which by (9) and (1) yields $\prob(\max_{\substack{S \subset V \\ |S| = k}} w(S)  \geq (2k-1-(4k-1)\varepsilon)\log{n}/n) \rightarrow 1$ as $n \rightarrow \infty$. This completes the proof of the lower bound and Theorem 3. \qed

\subsection{Theorem 4}

The upper bound of the result follows from Theorem 1 (ii) and Theorem 2. Take $S = \{ v_1, v_2, ... ,v_k \}$ to be our set of typical points, (which we can assume is non-empty else we are in the setting of Theorem \ref{thm:main}). We can use Theorem 1 (ii) to deduce that, $\forall \varepsilon > 0$, with probability tending to 1 as $n \to \infty$, \textbf{any} vertex in $V$ can be connected to $v_1$ (say) via a path of length at most $(2+\varepsilon)\log{n}/n$. Theorem 2 implies that with probability tending to 1 as $n \to \infty$, $S$ can be connected by a tree of weight at most $(k-1+\varepsilon)\log{n}/n$. Hence,

$$\prob(W_{k,l} > (k+2l-1+(l+1)\varepsilon)\log{n}/n) \to 0 \text{ as } n \to \infty$$

The lower bound can be proved using the same argument that appears in section 2.2. We simply need to ensure the set of $k$ ``typical vertices'' $\{ v_1, v_2, ..., v_k \}$ lies in $B$ when defining our sets $A_i, 1 \leq i \leq l$.
\\

\textbf{Acknowledgment}
\\

We thank B\'alint T\'oth for stimulating discussions which lead to our interest in this problem.



\bibliography{Maximal_Steiner_Trees}

\begin{thebibliography}{1}

\bibitem{bhamidi2013diameter}
Shankar Bhamidi and Remco van~der Hofstad.
\newblock Diameter of the stochastic mean-field model of distance.
\newblock {\em arXiv preprint arXiv:1306.0208}, 2013.

\bibitem{bhamidi2011first}
Shankar Bhamidi, Remco Van~der Hofstad, and Gerard Hooghiemstra.
\newblock First passage percolation on the {E}rd{\H{o}}s--{R}{\'e}nyi random
  graph.
\newblock {\em Combinatorics, Probability and Computing}, 20(05):683--707,
  2011.

\bibitem{bollobas2004value}
B{\'e}la Bollob{\'a}s, David Gamarnik, Oliver Riordan, and Benny Sudakov†.
\newblock On the value of a random minimum weight {S}teiner tree.
\newblock {\em Combinatorica}, 24(2):187--207, 2004.

\bibitem{frieze2004random}
Alan Frieze.
\newblock On random symmetric travelling salesman problems.
\newblock {\em Mathematics of Operations Research}, 29(4):878--890, 2004.

\bibitem{frieze1985value}
Alan~M Frieze.
\newblock On the value of a random minimum spanning tree problem.
\newblock {\em Discrete Applied Mathematics}, 10(1):47--56, 1985.

\bibitem{janson1999one}
Svante Janson.
\newblock One, two and three times log n/n for paths in a complete graph with
  random weights.
\newblock {\em Combinatorics, Probability and Computing}, 8(04):347--361, 1999.

\bibitem{van2006size}
Remco Van Der~Hofstad, Gerard Hooghiemstra, and Piet Van~Mieghem.
\newblock Size and weight of shortest path trees with exponential link weights.
\newblock {\em Combinatorics, Probability and Computing}, 15(06):903--926,
  2006.

\bibitem{wastlund2010mean}
Johan W{\"a}stlund.
\newblock The mean field traveling salesman and related problems.
\newblock {\em Acta mathematica}, 204(1):91--150, 2010.

\end{thebibliography}
\bibliographystyle{plain}

\end{document}